\documentclass[12pt,twoside]{amsart}
\usepackage{amsmath}
\usepackage{amsthm}
\usepackage{amsfonts}
\usepackage{amssymb}
\usepackage{latexsym}
\usepackage{mathrsfs}
\usepackage{amsmath}
\usepackage{amsthm}
\usepackage{amsfonts}
\usepackage{amssymb}
\usepackage{latexsym}
\usepackage{geometry}
\usepackage{dsfont}
\usepackage[dvips]{graphicx}
\usepackage{color}
\usepackage[all]{xy}

\date{}
\allowdisplaybreaks[4] \footskip=15pt
\renewcommand{\uppercasenonmath}[1]{}

\usepackage{graphicx,amssymb}
\usepackage[all]{xy}
\usepackage{amsmath}

\allowdisplaybreaks
\usepackage{amsthm}
\usepackage{color}

\theoremstyle{plain}
\newtheorem{theorem}{Theorem}[section]
\newtheorem{proposition}[theorem]{Proposition}

\newtheorem{corollary}[theorem]{Corollary}
\newtheorem{example}[theorem]{Example}
\newtheorem*{open question}{Open Question}
\newtheorem{definition}[theorem]{Definition}

\theoremstyle{definition}
\newtheorem*{acknowledgement}{Acknowledgement}

\theoremstyle{remark}
\newtheorem{remark}[theorem]{Remark}

\newcommand{\Tor}{\mbox{\rm Tor}}

\def\p{\frak p}
\def\m{\frak m}

\def\Hom{{\rm Hom}}

\def\Tor{{\rm Tor}}

\def\Ker{{\rm Ker}}

\def\Im{{\rm Im}}
\def\Coker{{\rm Coker}}

\def\Ann{{\rm Ann}}

\def\Max{{\rm Max}}

\def\Spec{{\rm Spec}}
\def\Max{{\rm Max}}

\begin{document}
\begin{center}
{\large  \bf Characterizing $S$-flat modules and $S$-von Neumann regular rings by uniformity}

\vspace{0.5cm}   Xiaolei Zhang$^{a}$

{\footnotesize 

E-mail: zxlrghj@163.com\\}
\end{center}

\bigskip
\centerline { \bf  Abstract}
\bigskip
\leftskip10truemm \rightskip10truemm \noindent

Let $R$ be a ring and $S$  a multiplicative subset of $R$. An $R$-module $T$ is called $u$-$S$-torsion ($u$- always abbreviates uniformly) provided that $sT=0$ for some $s\in S$.  The notion of  $u$-$S$-exact sequences is also introduced from the viewpoint of uniformity.  An $R$-module $F$ is called  $u$-$S$-flat provided that the induced sequence $0\rightarrow A\otimes_RF\rightarrow B\otimes_RF\rightarrow C\otimes_RF\rightarrow 0$ is  $u$-$S$-exact for any  $u$-$S$-exact sequence $0\rightarrow A\rightarrow B\rightarrow C\rightarrow 0$.  A ring $R$ is called $u$-$S$-von Neumann regular  provided there exists an element  $s\in S$ satisfying that for any $a\in R$ there exists  $r\in R$ such that $sa=ra^2$. We obtain that a ring $R$ is a $u$-$S$-von Neumann regular ring if and only if any $R$-module is  $u$-$S$-flat. Several properties of   $u$-$S$-flat modules and  $u$-$S$-von Neumann regular rings are  obtained.
\vbox to 0.3cm{}\\
{\it Key Words:} $u$-$S$-torsion module,  $u$-$S$-exact sequence,  $u$-$S$-flat module, $u$-$S$-von Neumann regular ring.\\
{\it 2010 Mathematics Subject Classification:}  13C12, 16D40, 16E50.

\leftskip0truemm \rightskip0truemm
\bigskip

\section{Introduction}
Throughout this article, $R$ is always a commutative ring with identity and $S$ is always  a multiplicative subset of $R$, that is, $1\in S$ and $s_1s_2\in S$ for any $s_1\in S, s_2\in S$.  Let $S$  be a multiplicative subset of $R$. Recall from  \cite[Definition 1.6.10]{fk16} that an $R$-module $M$ is called an \emph{$S$-torsion} module if for any $m\in M$, there is an $s\in S$ such that $sm=0$. $S$-torsion-free modules can be defined as the right part of the hereditary torsion theory $\tau_S$ generated by $S$-torsion modules (see \cite{S75}).  Early in 1965, N\v{a}st\v{a}sescu et al. \cite{N65} defined \emph{$\tau_S$-Noetherian rings} as rings $R$ satisfying that for any ideal $I$ of $R$ there is a finitely generated sub-ideal $J$ of $I$ such that $I/J$ is $S$-torsion. However, to  tie together some Noetherian properties of commutative rings and their polynomial rings or  formal power series rings, Anderson and Dumitrescu \cite{ad02} defined \emph{$S$-Noetherian rings} $R$, that is, any ideal of $R$ is $S$-finite in 2002. Recall from \cite{ad02} that an $R$-module $M$ is called \emph{$S$-finite} provided that $sM\subseteq F$ for some $s\in S$ and some finitely generated submodule $F$ of $M$. One can see that there is some  uniformity is hidden in the definition of $S$-finite modules. In fact, an $R$-module $M$ is  $S$-finite if and only if $s(M/F)=0$ for some $s\in S$ and some finitely generated submodule $F$ of $M$. In this article, we introduce the notion of \emph{$u$-$S$-torsion} modules $T$ for which there exists  $s\in S$ such that $sT=0$. The  notion of  $u$-$S$-torsion modules is different from that of  $S$-torsion modules (see Example \ref{exam-not-ut}). In the past few years,  the notions of $S$-analogues of Noetherian rings, coherent rings, almost perfect rings and strong Mori domains are introduced and studied extensively in \cite{ad02,s19,bh18,l15,lO14,kkl14}.

In this article, we introduce the notions of  $u$-$S$-monomorphisms,   $u$-$S$-epimorphisms,   $u$-$S$-isomorphisms and   $u$-$S$-exact sequences according to the idea of uniformity (see Definition \ref{s-exact-sequ}). Some properties of $u$-$S$-torsion modules and $S$-finite modules with respect to  $u$-$S$-exact sequences are given in Proposition \ref{s-exct-tor} and Proposition \ref{s-finite-exact}. We say an $R$-module $F$ is \emph{$u$-$S$-flat} provided that the induced sequence $0\rightarrow A\otimes_RF\rightarrow B\otimes_RF\rightarrow C\otimes_RF\rightarrow 0$ is  $u$-$S$-exact for any  $u$-$S$-exact sequence $0\rightarrow A\rightarrow B\rightarrow C\rightarrow 0$ (see Definition \ref{def-s-f}). Some basic characterizations of  $u$-$S$-flat modules are given (see Theorem \ref{s-flat-ext}).   It is well known that  an $R$-module $F$ is flat if and only $\Tor_1^R(R/I,F)=0$ for any ideal $I$ of $R$. However, the $S$-analogue of this result is not true (see Example \ref{uf not-extsion}). It is also worth  remarking that the class of   $u$-$S$-flat modules is not closed under direct limits and direct sums (see Remark \ref{s-dirlimt-not}). If an $R$-module $F$ is  $u$-$S$-flat, then $F_S$ is flat over $R_S$ (see Corollary \ref{s-flat-loc}). However, the converse does not hold (see Remark \ref{s-flat-not}). A new local characterization of flat modules is given in Proposition \ref{s-flat-loc-char}. A ring $R$ is called a $u$-$S$-von Neumann regular ring if there exists an element  $s\in S$ satisfies that for any $a\in R$ there exists  $r\in R$ such that $sa=ra^2$ (see Definition \ref{def-s-vn}).  A ring $R$ is \emph{$u$-$S$-von Neumann regular} if and only if  any $R$-module is $u$-$S$-flat (see Theorem \ref{s-vn-ext-char}).  Every $u$-$S$-von Neumann regular ring is locally von Neumann regular at $S$ (see Corollary \ref{s-local-vn}). However, the converse is also not true in general (see Example \ref{S-vn-not}). We also give a non-trivial example of $u$-$S$-von Neumann regular which is not von Neumann regular (see Example \ref{exam-not-ut-1}). Finally, we give a new local characterization of von Neumann regular rings in Proposition \ref{s-vn-loc-char}.

\section{$u$-$S$-torsion modules}

Recall from  \cite[Definition 1.6.10]{fk16} that an $R$-module $T$ is said to be an $S$-torsion module if for any $t\in T$ there is an element $s\in S$ such that $st=0$. Note that the choice of $s$ is decided by the element $t$. In this article, we care more about the uniformity  of $s$ on $T$.

\begin{definition} Let $R$ be a ring and $S$ a multiplicative subset of $R$. An $R$-module $T$ is called a $u$-$S$-torsion $($abbreviates uniformly $S$-torsion$)$ module  provided that there exists an element $s\in S$ such that $sT=0$.
\end{definition}

Obviously, the submodules and quotients of $u$-$S$-torsion modules are also $u$-$S$-torsion. Note that  finitely generated $S$-torsion modules are $u$-$S$-torsion and any $u$-$S$-torsion modules are  $S$-torsion. However,  $S$-torsion modules are not necessary $u$-$S$-torsion. We also note that every $R$-module does not have a maximal $u$-$S$-torsion submodule.
\begin{example}\label{exam-not-ut}
Let $\mathbb{Z}$ be the ring of integers, $p$  a prime in $\mathbb{Z}$ and  $S=\{p^n\mid n\geq 0\}$.  Let $M=\mathbb{Z}_{(p)}/\mathbb{Z}$ be a $\mathbb{Z}$-module where $\mathbb{Z}_{(p)}$ is the localization of $\mathbb{Z}$ at $S$. Then
\begin{enumerate}
\item  $M$ is  $S$-torsion but not $u$-$S$-torsion.
 \item  $M$ has no maximal $u$-$S$-torsion submodule.
\end{enumerate}
\end{example}
\begin{proof} (1) Obviously, $M$ is an $S$-torsion module. Suppose there is an $p^n$ such that $p^nM=0$. However,  $p^n(\frac{1}{p^{n+1}}+\mathbb{Z})=\frac{1}{p}+\mathbb{Z}\not= 0+\mathbb{Z}$ in $M$. Thus $M$ is not $u$-$S$-torsion.

(2)  Suppose  $N$ is a maximal $u$-$S$-torsion submodule of $M$. Then there is an element $p^n\in S$ such that $p^n N=0$. Note  $N$ is a submodule of $M_n:=\{\frac{a}{p^n}+\mathbb{Z}\in M\mid  a\in \mathbb{Z}\}$. Since $M_{n+1}:=\{\frac{a}{p^{n+1}}+\mathbb{Z}\in M\mid  a\in \mathbb{Z}\}$ is a  $u$-$S$-torsion submodule of $M$ and $N$ is a proper submodule of $M_{n+1}$. It is a contradiction.
\end{proof}

\begin{proposition}\label{fin-S-u-tor} Let $R$ be a ring and $M$ an $R$-module. Let $S$  be a  multiplicative subset of $R$ consisting of finite elements. Then $M$ is  $S$-torsion if and only if $M$ is $u$-$S$-torsion.
\end{proposition}
\begin{proof} If $M$ is $u$-$S$-torsion, then $M$ is trivially  $S$-torsion.  Let $S=\{s_1,...,s_n\}$ and $s=s_1...s_n$.  Suppose $M$ is an  $S$-torsion module. Then for any $m\in M$, there is an element $s_i\in S$ such that $s_im=0$.  Thus $sm=0$ for any $m\in M$. So $sM=0$.
\end{proof}

\begin{proposition} Let $R$ be a ring and $S$ a multiplicative subset of $R$. If an $R$-module $M$ has a maximal $u$-$S$-torsion submodule, then  $M$ has only one maximal $u$-$S$-torsion submodule.
\end{proposition}
\begin{proof} Let $M_1$ and $M_2$ be maximal $u$-$S$-torsion submodules of $M$ such that $s_1M_1=0$ and $s_2M_2=0$ for some $s_1,s_2\in S$. We claim that $M_1=M_2$. Indeed, otherwise we may assume there is an $m\in M_2-M_1$. Let $M_3$ be a submodule of $M$  generated by  $M_1$ and $m$.  Then $s_1s_2M_3=0$. Thus $M_3$ is a $u$-$S$-torsion submodule properly containing $M_1$, which is a contradiction.
\end{proof}

Recall from  \cite[Definition 1.6.10]{fk16} that an $R$-module $M$ is said to be an \emph{$S$-torsion-free module} if  $sm=0$ for some $s\in S$ and $m\in M$ implies that $m=0$. The classes of $S$-torsion modules and  $S$-torsion-free modules constitute a hereditary torsion theory (see \cite{S75}). From this result it follows immediately the next result (see \cite[Theorem 6.1.6]{fk16}). However we give direct proof for completeness.

\begin{proposition}  Let $R$ be a ring and $S$ a multiplicative subset of $R$.
Then an $R$-module $F$ is $S$-torsion-free if and only if $\Hom_R(T,F)=0$ for any $u$-$S$-torsion module $T$.
\end{proposition}
\begin{proof} Assume that $F$ is an $S$-torsion-free module and let $T$ be a $u$-$S$-torsion module  and $f\in \Hom_R(T,F)$. Then there exists  $s\in S$ such that $sT=0$. Thus for any  $t\in T$,  $sf(t)=f(st)=0\in F$. Thus $f(t)=0$ for any  $t\in T$. Conversely suppose that $sm=0$  for some $s\in S$  and $m\in F$. Set $F_s=\{x\in F\mid sx=0\}$. Then $F_s$ is a $u$-$S$-torsion submodule of $F$. Thus $\Hom_R(F_s,F)=0$. It follows that $F_s=0$ and thus $m$=0. So $F$ is $S$-torsion-free.
\end{proof}

\begin{corollary}\label{S-tor-tor} Let  $R$ be a ring, $S$ a multiplicative subset of $R$ and $T$  a $u$-$S$-torsion module. Then $\Tor^R_{n}(M,T)$ is $u$-$S$-torsion for any $R$-module $M$ and $n\geq 0$.
\end{corollary}
\begin{proof}  Let $T$ be a $u$-$S$-torsion module with $sT=0$. If $n=0$, then for any $\sum a\otimes b\in M\otimes_RT$, we have $s\sum a\otimes b=\sum a\otimes sb=0$. Thus $s(M\otimes_RT)=0$. Let $0\rightarrow\Omega (M)\rightarrow P\rightarrow M \rightarrow0$ be a short exact sequence with $P$ projective. Then $\Tor^R_{1}(M,T)$ is a submodule of $\Omega(M)\otimes_RT$ which is $u$-$S$-torsion. Thus $\Tor^R_{1}(M,T)$ is $u$-$S$-torsion. For $n\geq 2$, we have an isomorphism $\Tor^R_{n}(M,T)\cong \Tor^R_{1}(\Omega^{n-1}(M),T)$ where $\Omega^{n-1}(M)$ is the $(n-1)$-th syzygy of $M$. Since $\Tor^R_{1}(\Omega^{n-1}(M),T)$ is  $u$-$S$-torsion by induction, $\Tor^R_{n}(M,T)$ is  $u$-$S$-torsion.
\end{proof}

\begin{definition} \label{s-exact-sequ}
 Let $R$ be a ring and $S$ a multiplicative subset of $R$. Let $M$, $N$ and $L$ be $R$-modules.
\begin{enumerate}
\item An  $R$-homomorphism $f:M\rightarrow N$ is called a $u$-$S$-monomorphism $($resp.,    $u$-$S$-epimorphism$)$ provided that $\Ker(f)$ $($resp., $\Coker(f))$ is a  $u$-$S$-torsion module.

 \item    An  $R$-homomorphism $f:M\rightarrow N$ is  called a $u$-$S$-isomorphism  provided that $f$ is both a $u$-$S$-monomorphism and a $u$-$S$-epimorphism.

\item An $R$-sequence  $M\xrightarrow{f} N\xrightarrow{g} L$ is called  $u$-$S$-exact provided that there is an element $s\in S$ such that $s\Ker(g)\subseteq \Im(f)$ and $s\Im(f)\subseteq \Ker(g)$.
\end{enumerate}
\end{definition}
It is easy to verify that  $f:M\rightarrow N$ is a $u$-$S$-monomorphism  $($resp.,   $u$-$S$-epimorphism$)$ if and only if  $0\rightarrow M\xrightarrow{f} N$   $($resp., $M\xrightarrow{f} N\rightarrow 0$ $)$ is   $u$-$S$-exact.

\begin{proposition}\label{s-exct-tor}
Let $R$ be a ring, $S$ a multiplicative subset of $R$ and $M$ an $R$-module. Then the following assertions hold.
\begin{enumerate}
\item Suppose $M$ is $u$-$S$-torsion  and $f:L\rightarrow M$ is a $u$-$S$-monomorphism. Then $L$ is $u$-$S$-torsion.

\item  Suppose $M$ is $u$-$S$-torsion  and $g:M\rightarrow N$ is a $u$-$S$-epimorphism. Then $N$ is $u$-$S$-torsion.

\item Let $f:M\rightarrow N$ be a $u$-$S$-isomorphism. If one of  $M$ and $N$ is $u$-$S$-torsion, so is the other.

\item Let $0\rightarrow L\xrightarrow{f} M\xrightarrow{g} N\rightarrow 0$  be a $u$-$S$-exact sequence. Then  $M$ is  $u$-$S$-torsion if and only if $L$ and $N$ are $u$-$S$-torsion.
\end{enumerate}
\end{proposition}
\begin{proof} We only prove $(4)$ since $(1)$, $(2)$ and $(3)$ are the consequences of $(4)$.

 Suppose  $M$ is  $u$-$S$-torsion with $sM=0$.
Since $\Ker(f)$ (resp., $\Coker(g)$) is $u$-$S$-torsion with $s_1\Ker(f)=0$ (resp., $s_2\Coker(g)=0$)  for some $s_1\in S$ (resp., $s_2\in S$), it follows that $ss_1L=0$ (resp.,   $ss_2N=0$). Consequently,  $L$ (resp.,  $N$) is $u$-$S$-torsion. Now suppose $L$ and $N$ are $u$-$S$-torsion with $s_1L=s_2N=0$ for some $s_1, s_2\in S$. Since the  $u$-$S$-exact sequence is exact at $M$, there exists  $s\in S$ such that $s\Ker(g)\subseteq\Im(f)$ and $s\Im(f)\subseteq\Ker(g)$. Let $m\in M$. Then $s_2g(m)=g(s_2m)=0$. Thus there exists $l\in L$ such that $ss_2m= f(l)$. So $s_1ss_2m= s_1f(l)=f(s_1l)=0$.  So $M$ is  $u$-$S$-torsion.
\end{proof}

Let $R$ be a ring and $S$ a multiplicative subset of $R$. Recall from \cite{ad02} that an $R$-module $M$ is called \emph{$S$-finite}  provided that there exists  $s\in S$ such that $sM\subseteq N\subseteq M$, where $N$ is a finitely generated $R$-module. Let $M$ be an $R$-module,  $\{m_i\}_{i\in \Lambda}\subseteq M$ and $N=\langle m_i\rangle_{i\in \Lambda}$. We say an $R$-module $M$ is \emph{$S$-generated} by  $\{m_i\}_{i\in \Lambda}$ provided that $sM\subseteq N$ for some $s\in S$. Thus an $R$-module $M$ is  $S$-finite provided that $M$ can be $S$-generated by finite elements.

\begin{proposition}\label{s-finite-exact}
Let $R$ be a ring,  $S$ a multiplicative subset of $R$ and $M$ an $R$-module. Then the following assertions hold.
\begin{enumerate}
\item Let $M$ be an $S$-finite $R$-module and $f:M\rightarrow N$  a $u$-$S$-epimorphism. Then $N$ is $S$-finite.

\item Let  $0\rightarrow L\xrightarrow{f} M\xrightarrow{g} N\rightarrow 0$  be a $u$-$S$-exact sequence. If $L$ and $N$ are $S$-finite, so is $M$.

\item Let $f:M\rightarrow N$ be a $u$-$S$-isomorphism. If one of  $M$ and $N$ is $S$-finite, so is the other.
\end{enumerate}
\end{proposition}
\begin{proof} (1) Consider the exact sequence $M\xrightarrow{f} N\rightarrow T\rightarrow 0$ with $sT=0$ for some $s\in S$. Let $F$ be a finitely generated submodule of $M$ such that $s'M\subseteq F$ for some $s'\in S$. Then $f(F)$ is a finitely generated submodule of $N$ such that $ss'N\subseteq f(F)$.

(2) Suppose $0\rightarrow L\xrightarrow{f} M\xrightarrow{g} N\rightarrow 0$  is a $u$-$S$-exact sequence. Let $L_1$ and $N_1$ be finitely generated submodules of $L$ and $N$ such that $s_LL\subseteq L_1$ and $s_NN\subseteq N_1$ for some $s_L, s_N\in S$  respectively. Let $M_1$ be a finitely generated submodule of $M$ generated by the  finite images of generators  of $L_1$ and the finite pre-images of finite generators  of $N_1$. Then for any $m\in M$, $s_Ng(m)\in N_1$. Thus there exists $m_1\in M_1$ such that $s_Ng(m)=g(m_1)$. We have $s_Nm-m_1\in \Ker(g)$. Since there exists $s\in S$ such that $s\Ker(g)\subseteq \Im(f)$. So there exists $l\in L$ such that $s(s_Nm-m_1)=f(l)$. Then there exists $l_1\in L_1$ such that $s_Ll=l_1$. Thus  $s_Ls(s_Nm-m_1)=s_Lf(l)=f(s_Ll)=f(l_1)$. Consequently, $s_Lss_Nm=s_Lsm_1+sf(l_1)\in M_1$. So $s_Lss_NM\subseteq M_1$. Since $M_1$ is finitely generated, we have $M$ is $S$-finite.

(3) It is a consequence of $(2)$.

\end{proof}

\section{  $u$-$S$-flat modules and  $u$-$S$-von Neumann regular rings }
Recall from \cite{fk16} that an $R$-module $F$ is called \emph{flat} provided that for any short exact sequence $0\rightarrow A\rightarrow B\rightarrow C\rightarrow 0$, the induced sequence $0\rightarrow A\otimes_RF\rightarrow B\otimes_RF\rightarrow C\otimes_RF\rightarrow 0$ is exact. Now, we give an $S$-analogue of flat modules.
\begin{definition}\label{def-s-f} Let $R$ be a ring, $S$ a multiplicative subset of $R$.
An $R$-module $F$ is called  $u$-$S$-flat (abbreviates uniformly $S$-flat) provided that for any  $u$-$S$-exact sequence $0\rightarrow A\rightarrow B\rightarrow C\rightarrow 0$, the induced sequence $0\rightarrow A\otimes_RF\rightarrow B\otimes_RF\rightarrow C\otimes_RF\rightarrow 0$ is  $u$-$S$-exact .
\end{definition}

Recall from \cite{fk16} that an $R$-module $F$  is flat if and only if  $\Tor^R_1(M,F)=0$   for any  $R$-module $M$, if and only if $\Tor^R_n(M,F)=0$  for any  $R$-module $M$ and $n\geq 1$. We give an $S$-analogue of this result.

\begin{theorem}\label{s-flat-ext}
Let $R$ be a ring, $S$ a multiplicative subset of $R$ and $F$ an $R$-module. The following statements are equivalent:
\begin{enumerate}
\item  $F$ is  $u$-$S$-flat;

\item for any short exact sequence $0\rightarrow A\xrightarrow{f} B\xrightarrow{g} C\rightarrow 0$, the induced sequence $0\rightarrow A\otimes_RF\xrightarrow{f\otimes_RF} B\otimes_RF\xrightarrow{g\otimes_RF} C\otimes_RF\rightarrow 0$ is  $u$-$S$-exact;

\item  $\Tor^R_1(M,F)$ is  $u$-$S$-torsion for any  $R$-module $M$;

\item  $\Tor^R_n(M,F)$ is  $u$-$S$-torsion for any  $R$-module $M$ and $n\geq 1$.
\end{enumerate}
\end{theorem}
\begin{proof} $(1)\Rightarrow(2)$, $(3)\Rightarrow(2)$ and $(4)\Rightarrow(3)$: Trivial.

$(2)\Rightarrow(3)$:  Let $0\rightarrow L\rightarrow P\rightarrow M\rightarrow 0$ be a short exact sequence with $P$ projective. Then there exists a long exact sequence  $$0\rightarrow\Tor_1^R(M,F)\rightarrow F\otimes L\rightarrow P\otimes F\rightarrow M\otimes F\rightarrow 0.$$ Thus  $\Tor^R_1(M,F)$ is  $u$-$S$-torsion by $(2)$.

$(3)\Rightarrow(4)$:  Let $M$ be an $R$-module. Denote by $\Omega^{n-1}(M)$ the $(n-1)$-th syzygy of $M$. Then $\Tor^R_n(M,F)\cong \Tor^R_1(\Omega^{n-1}(M),F)$ is $u$-$S$-torsion by $(3)$.

$(2)\Rightarrow(1)$: Let $F$ be an $R$-module satisfies $(2)$.  Suppose $0\rightarrow A\xrightarrow{f} B\xrightarrow{g} C\rightarrow 0$ is a $u$-$S$-exact sequence. Then there is an exact sequence  $B\xrightarrow{g} C\rightarrow T\rightarrow 0 $ where $T=\Coker(g)$ is $u$-$S$-torsion. Tensoring $F$ over $R$, we have an exact sequence $$B\otimes_RF\xrightarrow{g\otimes_RF} C\otimes_RF\rightarrow T\otimes_RF\rightarrow 0.$$
Then $T\otimes_RF$ is $u$-$S$-torsion by Corollary \ref{S-tor-tor}. Thus $0\rightarrow A\otimes_RF\xrightarrow{f\otimes_RF} B\otimes_RF\xrightarrow{g\otimes_RF} C\otimes_RF\rightarrow 0$ is  $u$-$S$-exact at $C\otimes_RF$.

There are naturally two short exact sequences: $0\rightarrow \Ker(f)\rightarrow A\rightarrow \Im(f)\rightarrow 0$, $0\rightarrow \Im(f)\rightarrow B\rightarrow \Coker(f)\rightarrow 0$, where $\Ker(f)$ is $u$-$S$-torsion. Consider the  induced exact sequence  $$\rightarrow  \Ker(f)\otimes_RF\xrightarrow{i_{\Ker(f)}\otimes_RF}  A\otimes_RF\rightarrow \Im(f)\otimes_RF\rightarrow 0,$$
$$\rightarrow  \Im(f)\otimes_RF\xrightarrow{i_{\Im(f)}\otimes_RF}  B\otimes_RF\rightarrow \Coker(f)\otimes_RF\rightarrow 0,$$
where $\Ker(i_{\Im(f)}\otimes_RF)$  and $\Ker(i_{\Ker(f)}\otimes_RF)$ are $u$-$S$-torsion. We have the following pull-back diagram:

$$\xymatrix@R=20pt@C=25pt{ & 0\ar[d]&0\ar[d]&&\\
 & \Im(i_{\Ker(f)}\otimes_RF)\ar[d]\ar@{=}[r]^{} & \Im(i_{\Ker(f)}\otimes_RF)\ar[d]& &  \\
0 \ar[r]^{} & Y                      \ar[r]^{} \ar[d] & A\otimes_RF \ar[d]\ar[r]^{} &\Im(i_{\Im(f)}\otimes_RF)\ar@{=}[d]\ar[r]^{} & 0 \\
 0 \ar[r]^{}& \Ker(i_{\Im(f)}\otimes_RF)\ar[d]\ar[r]&\Im(f)\otimes_RF\ar[r]\ar[d]&\Im(i_{\Im(f)}\otimes_RF)\ar[r] &0\\
 & 0 &0 & &\\}$$
Since $\Ker(f)$ is $u$-$S$-torsion, so is $\Ker(f)\otimes_RF$ by Corollary \ref{S-tor-tor}. Hence  $\Im(i_{\Ker(f)}\otimes_RF)$ is $u$-$S$-torsion, and thus $Y$ is also $u$-$S$-torsion by Proposition \ref{s-exct-tor}. So the composition $f\otimes_RF: A\otimes_RF \twoheadrightarrow\Im(i_{\Im(f)}\otimes_RF)\rightarrowtail B\otimes_RF$ is a $u$-$S$-monomorphism. Thus $0\rightarrow A\otimes_RF\xrightarrow{f\otimes_RF} B\otimes_RF\xrightarrow{g\otimes_RF} C\otimes_RF\rightarrow 0$ is  $u$-$S$-exact at $A\otimes_RF$.

Since the sequence $0\rightarrow A\xrightarrow{f} B\xrightarrow{g} C\rightarrow 0$ is $u$-$S$-exact at $B$, there exists  $s_1\in S$ such that   $s_1\Ker(g)\subseteq \Im(f)$ and  $s_1\Im(f)\subseteq \Ker(g)$. By $(2)$, there are two exact sequences $0\rightarrow T_1\rightarrow s_1\Ker(g)\otimes_RF\rightarrow \Im(f)\otimes_RF$ with $s_2T_1=0$ for some $s_2\in S$, and  $0\rightarrow T_2\rightarrow s_1\Im(f)\otimes_RF\rightarrow \Ker(g)\otimes_RF$ with $s_3T_2=0$ for some $s_3\in S$. Consider the induced sequence $0\rightarrow T\rightarrow \Ker(g)\otimes_RF\rightarrow B\otimes_RF\rightarrow \Coker(g)\otimes_RF\rightarrow 0$ with $s_4T=0$ for some $s_4\in S$. Set $s=s_1s_2s_3s_4$, we will show $s\Ker(g\otimes_RF)\subseteq \Im(f\otimes_RF)$ and $s\Im(f\otimes_RF)\subseteq \Ker(g\otimes_RF)$.  Consider the following exact sequence $$0\rightarrow T\rightarrow\Ker(g)\otimes_RF\xrightarrow{i_{\Ker(g)}\otimes_RF} B\otimes_RF\xrightarrow{g\otimes_RF} C\otimes_RF.$$  Then $\Im(i_{\Ker(g)}\otimes_RF)=\Ker(g\otimes_RF)$. Thus $s\Ker(g\otimes_RF)= s_1s_2s_3s_4\Ker(g\otimes_RF)=s_1s_2s_3s_4\Im(i_{\Ker(g)}\otimes_RF)\subseteq s_1s_2s_3\Ker(g)\otimes_RF\subseteq s_3\Im(f)\otimes_RF=s_3\Im(f\otimes_RF)\subseteq \Im(f\otimes_RF),$ and $s \Im(f\otimes_RF)= s_1s_2s_3s_4\Im(f)\otimes_RF\subseteq s_2s_4\Ker(g)\otimes_RF\subseteq  s_2\Im(i_{\Ker(g)}\otimes_RF)=s_2\Ker(g\otimes_RF)\subseteq \Ker(g\otimes_RF)$. Thus $0\rightarrow A\otimes_RF\rightarrow B\otimes_RF\rightarrow C\otimes_RF\rightarrow 0$ is  $u$-$S$-exact at $ B\otimes_RF$.
\end{proof}

By Corollary \ref{S-tor-tor} and Theorem \ref{s-flat-ext}, flat modules and $u$-$S$-torsion modules are  $u$-$S$-flat. And $u$-$S$-flat modules are flat provided that any element in $S$ is a unit. Moreover, if  any element in $S$ is regular and all  $u$-$S$-flat modules are flat, then any element in $S$ is a unit. Indeed, for any $s\in S$, we have $R/\langle s\rangle$ is   $u$-$S$-flat and thus  flat. So $\langle s\rangle$ is a pure ideal of $R$. By \cite[Theorem 1.2.15]{g}, there exists  $r\in R$ such that $s(1-rs)=0$. Since $s$ is regular, $s$ is a unit.

The following example shows that the condition ``$\Tor^R_1(M,F)$ is  $u$-$S$-torsion for any  $R$-module $M$'' in Theorem \ref{s-flat-ext} can not be replaced by ``$\Tor^R_1(R/I,F)$ is  $u$-$S$-torsion for any  ideal $I$ of  $R$''.
\begin{example}\label{uf not-extsion}
Let $\mathbb{Z}$ be the ring of integers, $p$ a prime in $\mathbb{Z}$ and  $S=\{p^n\mid n\geq 0\}$ as in Example \ref{exam-not-ut}. Let $M=\mathbb{Z}_{(p)}/\mathbb{Z}$. Then $\Tor^R_1(R/I,M)$ is  $u$-$S$-torsion for any  ideal $I$ of  $R$. However, $M$ is not  $u$-$S$-flat.
\end{example}
\begin{proof}
Let $\langle n\rangle$ be an ideal of $\mathbb{Z}$.   It follows from \cite[Chapter I, Lemma 6.2(a)]{FS01} that $\Tor_1^{\mathbb{Z}}(\mathbb{Z}/\langle n\rangle, M)\cong \{m\in M\mid nm=0\}=\{\frac{b}{p^a}+\mathbb{Z}\in \mathbb{Z}_{(p)}/\mathbb{Z}\mid  a, b$ satisfies $p^a|nb\}$.   Write $n=p^km$ where $(p,m)=1$.
If $k=0$, then $\Tor_1^{\mathbb{Z}}(\mathbb{Z}/\langle n\rangle, M)=0$.  If $k\geq 1$, then $\Tor_1^{\mathbb{Z}}(\mathbb{Z}/\langle n\rangle, M)=\{\frac{b}{p^k}+\mathbb{Z}\in \mathbb{Z}_{(p)}/\mathbb{Z}\mid  a, b\in \mathbb{Z}\}$. Thus $p^k\cdot\Tor_1^{\mathbb{Z}}(\mathbb{Z}/\langle n\rangle, M)=0$. So $\Tor_1^{\mathbb{Z}}(\mathbb{Z}/\langle n\rangle, M)$ is $u$-$S$-torsion for any ideal $\langle n\rangle$ of $\mathbb{Z}$.  However,
 $\Tor_1^{\mathbb{Z}}(\mathbb{Q}/\mathbb{Z}, \mathbb{Z}_{(p)}/\mathbb{Z})\cong t(\mathbb{Z}_{(p)}/\mathbb{Z})= \mathbb{Z}_{(p)}/\mathbb{Z}$  by \cite[Chapter I, Lemma 6.2(b)]{FS01}. Since $\mathbb{Z}_{(p)}/\mathbb{Z}$ is not  $u$-$S$-torsion by Example \ref{exam-not-ut},
 $M=\mathbb{Z}_{(p)}/\mathbb{Z}$ is not  $u$-$S$-flat.
\end{proof}

\begin{proposition}\label{s-flat-exac}
Let $R$ be a ring and $S$ a multiplicative subset of $R$. Then the following statements hold.
\begin{enumerate}
\item Any pure quotient of  $u$-$S$-flat modules is  $u$-$S$-flat.
\item Any finite direct sum of  $u$-$S$-flat modules is  $u$-$S$-flat.
\item Let $0\rightarrow A\xrightarrow{f} B\xrightarrow{g} C\rightarrow 0$  be a $u$-$S$-exact sequence. If $A$ and $C$ are  $u$-$S$-flat modules, so is $B$.
\item  Let $A\rightarrow B$ be a $u$-$S$-isomorphism. If one of $A$ and $B$ is  $u$-$S$-flat, so is the other.
\item Let $0\rightarrow A\xrightarrow{f} B\xrightarrow{g} C\rightarrow 0$  be a $u$-$S$-exact sequence.  If $B$ and $C$ are  $u$-$S$-flat, then $A$ is  $u$-$S$-flat.
\end{enumerate}
\end{proposition}
\begin{proof} $(1)$ Let $0\rightarrow A\rightarrow B\rightarrow C\rightarrow 0$ be a pure exact sequence with $B$ $u$-$S$-flat.
Let $M$ be an $R$-module. Then there is an exact sequence  $ \Tor_1^R(M,B) \rightarrow \Tor_1^R(M,C)\rightarrow 0$. Since $\Tor_1^R(M,B)$ is $u$-$S$-torsion, $\Tor_1^R(M,C)$ also  is $u$-$S$-torsion. Thus $C$ is  $u$-$S$-flat.

$(2)$ Let $F_1,...,F_n$ be  $u$-$S$-flat modules. Let $M$ be an $R$-module. Then there exists  $s_i\in S$ such that  $s_i\Tor_1^R(M,F_i)=0$. Set $s=s_1...s_n$. Then $s\Tor_1^R(M,\bigoplus\limits_{i=1}^n F_i)\cong \bigoplus\limits_{i=1}^ns\Tor_1^R(M, F_i)=0$. Thus $\bigoplus\limits_{i=1}^n F_i$ is  $u$-$S$-flat.

$(3)$ Let $0\rightarrow A\xrightarrow{f} B\xrightarrow{g} C\rightarrow 0$ be a $u$-$S$-exact sequence. Then there are three short exact sequences: $0\rightarrow \Ker(f)\rightarrow A\rightarrow \Im(f)\rightarrow 0$, $0\rightarrow \Ker(g)\rightarrow B\rightarrow \Im(g)\rightarrow 0$ and $0\rightarrow \Im(g)\rightarrow C\rightarrow \Coker(g)\rightarrow 0$. Then  $\Ker(f)$ and $\Coker(g)$ are all $u$-$S$-torsion  and $s\Ker(g)\subseteq \Im(f)$ and $s\Im(f)\subseteq \Ker(g)$ for some $s\in S$.  Let $M$ be an $R$-module. Suppose $A$ and $C$ are  $u$-$S$-flat.  Then $$\Tor_1^R(M,A)\rightarrow \Tor_1^R(M,\Im(f)) \rightarrow M\otimes_R\Ker(f)$$ is exact. Since $\Ker(f)$ is $u$-$S$-torsion and $A$ is  $u$-$S$-flat, it follows that  $\Tor_1^R(M,\Im(f))$ is $u$-$S$-torsion.
Note $$\Tor_2^R(M,\Coker(g))\rightarrow \Tor_1^R(M,\Im(g))\rightarrow \Tor_1^R(M,C)$$  is exact. Since $\Coker(g)$ is   $u$-$S$-torsion, then $\Tor_2^R(M,\Coker(g))$ is  $u$-$S$-torsion by Corollary \ref{S-tor-tor}. Thus $\Tor_1^R(M,\Im(g))$ is  $u$-$S$-torsion as $\Tor_1^R(M,C)$ is $u$-$S$-torsion. We also note that
$$\Tor_1^R(M,\Ker(g))\rightarrow \Tor_1^R(M,B) \rightarrow \Tor_1^R(M,\Im(g))$$ is exact. Thus to verify $\Tor_1^R(M,B)$ is $u$-$S$-torsion, we just need to show $\Tor_1^R(M,\Ker(g))$ is $u$-$S$-torsion. Set $N= \Ker(g)+\Im(f)$.
Consider the following two exact sequences
\begin{center}
$0\rightarrow \Ker(g)\rightarrow N\rightarrow N/\Ker(g)\rightarrow 0$ and $0\rightarrow \Im(f)\rightarrow N\rightarrow  N/\Im(f)\rightarrow 0.$
\end{center}
Then it is easy to verify $N/\Ker(g)$ and $N/\Im(f)$ are all  $u$-$S$-torsion. Consider the following induced two exact sequences
$$\Tor_2^R(M,N/\Im(f))\rightarrow \Tor_1^R(M,\Ker(g)) \rightarrow \Tor_1^R(M,N)  \rightarrow \Tor_1^R(M, N/\Im(f)),$$ $$\Tor_2^R(M,N/\Ker(g)) \rightarrow \Tor_1^R(M,\Im(f)) \rightarrow \Tor_1^R(M,N) \rightarrow \Tor_1^R(M, N/\Ker(g)).$$ Thus $\Tor_1^R(M,\Ker(g))$ is $u$-$S$-torsion if and only if  $\Tor_1^R(M,\Im(f))$ is $u$-$S$-torsion. Consequently,  $B$ is  $u$-$S$-flat since $\Tor_1^R(M,\Im(f))$ is proved to be $u$-$S$-torsion as above.

$(4)$ It can be certainly deduced from $(3)$.

(5) Let $0\rightarrow A\xrightarrow{f} B\xrightarrow{g} C\rightarrow 0$ be a $u$-$S$-exact sequence. Then, as in the proof of $(3)$, there are three short exact sequences: $0\rightarrow \Ker(f)\rightarrow A\rightarrow \Im(f)\rightarrow 0$, $0\rightarrow \Ker(g)\rightarrow B\rightarrow \Im(g)\rightarrow 0$ and $0\rightarrow \Im(g)\rightarrow C\rightarrow \Coker(g)\rightarrow 0$. Then  $\Ker(f)$ and $\Coker(g)$ are all $u$-$S$-torsion  and $s\Ker(g)\subseteq \Im(f)$ and $s\Im(f)\subseteq \Ker(g)$ for some $s\in S$. Let $M$ be an $R$-module. Note that $$\Tor_1^R(M,\Ker(f))\rightarrow\Tor_1^R(M,A)\rightarrow \Tor_1^R(M,\Im(f)) \rightarrow M\otimes_R\Ker(f)$$ is exact. Since $\Ker(f)$ is $u$-$S$-torsion,  then $\Tor_1^R(M,\Ker(f))$ and $M\otimes_R\Ker(f)$ are  $u$-$S$-torsion by Corollary \ref{S-tor-tor}. It just need to verify $\Tor_1^R(M,\Im(f))$ is  $u$-$S$-torsion. By the proof of $(3)$, we just need to show $\Tor_1^R(M,\Ker(g))$ is  $u$-$S$-torsion.  Since  $$ \Tor_2^R(M,\Im(g))\rightarrow \Tor_1^R(M,\Ker(g))\rightarrow \Tor_1^R(M,B)$$ is exact and $\Tor_1^R(M,B)$  is $u$-$S$-torsion, we just need to show $\Tor_2^R(M,\Im(g))$ is $u$-$S$-torsion. Note that $$\Tor_3^R(M,\Coker(g))\rightarrow  \Tor_2^R(M,\Im(g))\rightarrow \Tor_2^R(M,C)$$  is exact. Since $\Coker(g)$ is $u$-$S$-torsion and $C$ is $u$-$S$-flat, we have $\Tor_3^R(M,\Coker(g)) $ and $\Tor_2^R(M,C)$ are $u$-$S$-torsion. So  $\Tor_2^R(M,\Im(g))$ is $u$-$S$-torsion.
\end{proof}

\begin{remark}\label{s-dirlimt-not}  It is well known that any direct limit of flat modules is flat. However, every direct limit of  $u$-$S$-flat  modules is not  $u$-$S$-flat. Let $\mathbb{Z}$ be the ring of integers, $p$ a prime in $\mathbb{Z}$ and  $S=\{p^n\mid n\geq 0\}$ as in Example \ref{uf not-extsion}.  Let $F_n=\mathbb{Z}/\langle p^n\rangle$ be a $\mathbb{Z}$-module. Then $F_n$ is $u$-$S$-torsion, and thus   $u$-$S$-flat. Note that each $F_n$ is isomorphic to $M_n=\{\frac{a}{p^n}+\mathbb{Z}\in \mathbb{Z}_{(p)}/\mathbb{Z}\mid a\in \mathbb{Z}\}$. It is easy to verify $\mathbb{Z}_{(p)}/\mathbb{Z}=\bigcup\limits_{i=1}^{\infty}M_n \cong\lim\limits_{\longrightarrow}F_n$.  However,  $\mathbb{Z}_{(p)}/\mathbb{Z}$ is not  $u$-$S$-flat (see Example \ref{uf not-extsion}).

It is also worth noting infinite direct sums of  $u$-$S$-flat  modules need not be $u$-$S$-flat. Let $M_n=\{\frac{a}{p^n}+\mathbb{Z}\in \mathbb{Z}_{(p)}/\mathbb{Z}\mid a\in \mathbb{Z}\}$ as above. Then $M_n$ is  $u$-$S$-flat. Set $N=\bigoplus\limits_{n=1}^{\infty}M_n$. Then $N$ is a torsion module. Thus $\Tor_1^{\mathbb{Z}}(\mathbb{Q}/\mathbb{Z},N)=N $  by \cite[Chapter I, Lemma 6.2(b)]{FS01}. It can similarly be deduced from the proof of Example \ref{exam-not-ut} that $N$ is not $u$-$S$-torsion. Thus $N$ is not  $u$-$S$-flat.
\end{remark}

\begin{corollary}\label{s-flat-loc}
Let $R$ be a ring and $S$ a multiplicative subset of $R$. If $F$ is  $u$-$S$-flat over a ring $R$, then $F_S$ is flat over $R_S$.
\end{corollary}
\begin{proof} Let $I_S$ be a finitely generated ideal of $R_S$, where $I$ is a finitely generated ideal of $R$. Then there exists  $s\in S$ such that  $s\Tor_1^R(R/I,F)=0.$ Thus $0=\Tor_1^R(R/I,F)_S\cong \Tor_1^{R_S}(R_S/I_S,F_S)$. So $F_S$ is flat over $R_S$.
\end{proof}


\begin{remark}\label{s-flat-not}
Note that the converse of Corollary \ref{s-flat-loc} does not hold. Consider  $\mathbb{Z}$-module $M=\mathbb{Z}_{(p)}/\mathbb{Z}$ in  Example \ref{exam-not-ut}. Let $S=\{p^n\mid n\geq 0\}$. Then $M_S=0$ and thus is  flat over $\mathbb{Z}_S$. However, $M$ is not  $u$-$S$-flat over $\mathbb{Z}$ (see  Example \ref{uf not-extsion}).
\end{remark}

\begin{proposition} Let $R$ be a ring and $F$ an $R$-module. Let $S$ be a  multiplicative subset of $R$ consisting of finite  elements. Then $F$ is  $u$-$S$-flat over a ring $R$ if and only if  $F_S$ is flat over $R_S$.
\end{proposition}
\begin{proof} We just need to show that if $F_S$ is flat over $R_S$, then $F$ is  $u$-$S$-flat over a ring $R$. Let $0\rightarrow A\xrightarrow{f} B\rightarrow C\rightarrow 0$ be a short exact sequence over $R$. By tensoring $F$, we have an exact sequence $0\rightarrow T\rightarrow A\otimes_RF\xrightarrow{f\otimes_RF} B\otimes_RF\rightarrow C\otimes_RF\rightarrow 0$ where $T$ is the kernel of $f\otimes_RF$. By tensoring $R_S$, we have an exact sequence $0\rightarrow T_S\rightarrow A_S\otimes_{R_S}F_S\rightarrow B_S\otimes_{R_S}F_S\rightarrow C_S\otimes_{R_S}F_S\rightarrow 0$ over $R_S$. Since $F_S$ is flat over $R_S$, $T_S=0$. Thus $T$ is $S$-torsion. By Proposition \ref{fin-S-u-tor}, $T$ is $u$-$S$-torsion. So $F$ is  $u$-$S$-flat over a ring $R$.
\end{proof}

Let $\p$ be a prime ideal of $R$. We say an $R$-module $F$ is \emph{$u$-$\p$-flat} shortly provided that  $F$ is  $u$-$(R\setminus\p)$-flat.
\begin{proposition}\label{s-flat-loc-char}
Let $R$ be a ring and $F$ an $R$-module. Then the following statements are equivalent:
 \begin{enumerate}
\item  $F$ is flat;
\item   $F$ is    $u$-$\p$-flat for any $\p\in \Spec(R)$;
\item   $F$ is   $u$-$\m$-flat for any $\m\in \Max(R)$.
 \end{enumerate}
\end{proposition}
\begin{proof} $(1)\Rightarrow (2)\Rightarrow (3):$  Trivial.

 $(3)\Rightarrow (1):$ Let $M$ be an $R$-module. Then $\Tor_1^R(M,F)$ is $(R\setminus\m)$-torsion. Thus for any $\m\in \Max(R)$, there exists  $s_{\m}\in S$ such that $s_{\m}\Tor_1^R(M,F)=0$. Since the ideal generated by all $s_{\m}$ is $R$, $\Tor_1^R(M,F)=0$. So $F$ is flat.
\end{proof}

Recall that a ring $R$ is called \emph{von Neumann regular} provided that for any $a\in R$, there exists  $r\in R$ such that $a=ra^2$. One of the main topics  is the $S$-analogue of von Neumann regular rings. In order to study further, we will  characterize when a ring $R_S$ is  von Neumann regular in the next result.

\begin{proposition}\label{s-inj-ext}
Let $R$ be a ring and $S$ a multiplicative subset of $R$. The following statements are equivalent:
\begin{enumerate}
\item  $R_S$ is a von Neumann regular ring;
\item   any principal ideal of $R$ is $S$-generated by an idempotent;
\item   any $S$-finite ideal of $R$ is $S$-generated by an idempotent;
\item  for any $a\in R$, there exist  $s\in S$ and  $r\in R$ such that $sa=ra^2$;
\item   any $R_S$-module is flat over $R_S$.
\end{enumerate}
\end{proposition}
\begin{proof}

$(1)\Leftrightarrow (5):$ It is well known. $(3)\Rightarrow (2):$ Trivial.

$(1)\Rightarrow (4):$ Let $a\in R$. Then there exists $\frac{r_1}{s_1}\in R_S$ such that $\frac{a}{1}=\frac{r_1}{s_1}\frac{a^2}{1}$. Thus there  exists  $s_2\in S$ such that $s_1s_2a=s_2r_1a^2$.  Set $s=s_1s_2$ and $r=s_2r_1$, $(4)$ holds naturally.

$(4)\Rightarrow (1):$  Let $\frac{a}{s}$ be an element in $R_S$. Then there are $s'\in S$ and $x\in R$ such that $s'a=xa^2$. Thus $\frac{a}{s}=\frac{sx}{s'}(\frac{a}{s})^2$. So  $R_S$ is a von Neumann regular ring.


$(4)\Rightarrow (2):$ Let $\langle a\rangle$ be a principal ideal of $R$. Then there exists  $s\in S$ such that $sa=ra^2$ for some $r\in R$. Set $e=ra$. Then $se=e^2$ and  $e\in \langle a\rangle$. Since $sa=ea\in \langle e\rangle$, we have $s\langle a\rangle\subseteq \langle e\rangle\subseteq \langle a\rangle$.

$(2)\Rightarrow (3):$ Let $K$ be an  $S$-finite ideal and   $I=Ra_1+\cdots+Ra_n$ be a finitely generated sub-ideal of $I$ such that $s'K\subseteq I$ for some $s'\in S$. By $(2)$, for each $i$ there is an  idempotent $e_i\in Ra_i$ such that $s_i\langle a_i\rangle\subseteq \langle e_i\rangle$ for some $s_i\in S$ $(i=1,\cdots ,n)$.
Set $s=s's_1\cdots s_n$.  Then $s\langle a_i\rangle\subseteq \langle e_i\rangle$. Set $J=Re_1+\cdots +Re_n$.  Then $J$ is a sub-ideal of $I$ (thus of $K$) such that $sK\subseteq s_1\cdots s_n I \subseteq J$. Claim that $J$ is generated by an idempotent. Indeed, for any $x\in J$, we have $x=r_1e_1+\cdots +r_ne_n=r_1e^2_1+\cdots +r_ne^2_n\in J^2$. Thus $J^2=J$. Since $J$ is finitely generated, $J=\langle e\rangle$ for some idempotent $e\in I$ by \cite[Theorem 1.8.22]{fk16}.

$(2)\Rightarrow (4):$ Let $a\in R$. Then there is an idempotent $e$ such that $s\langle a\rangle \subseteq \langle e\rangle \subseteq \langle a\rangle$. If $e=ba$ for some $b \in R$, then $e=e^2=b^2a^2$. Thus $sa=ce=cb^2a^2$ for some $cb^2\in R$. So $(4)$ holds.
\end{proof}

Recall from \cite{bh18} that a ring $R$ is called $c$-$S$-coherent if any $S$-finite ideal $I$ is $c$-$S$-finitely presented, that is, there exists a finitely presented sub-ideal $J$ of $I$ such that $sI\subseteq J\subseteq I$. By Proposition \ref{s-inj-ext}, the following result holds since any ideal generated by an idempotent is projective, and thus is finitely presented .

\begin{corollary}\label{c-s-coh} Let $R$ be a ring and $S$ a multiplicative subset of $R$. If  $R_S$ is a von Neumann regular ring, then $R$ is $c$-$S$-coherent.
\end{corollary}

It is certain that for a ring $R$ such that $R_S$ is von Neumann regular,  the element $s\in S$ such that $sa=ra^2$ for some $r\in R$ depends on $a\in R$ by Proposition \ref{s-inj-ext}. Now we give the definition of $u$-$S$-von Neumann regular ring for which the element $s\in S$ is uniform on any element $a\in R$.

\begin{definition}\label{def-s-vn} Let $R$ be a ring and $S$ a multiplicative subset of $R$.
 $R$ is called a $u$-$S$-von Neumann regular ring $($abbreviates uniformly $S$-von Neumann regular ring$)$ provided there exists an element  $s\in S$ satisfying that for any $a\in R$ there exists  $r\in R$ such that $sa=ra^2$.
\end{definition}

Let $\{M_j\}_{j\in \Gamma}$  be a family of $R$-modules. Let $\{m_{i,j}\}_{i\in \Lambda_j}\subseteq M_j$ for each $j\in \Gamma$ and $N_j=\langle m_{i,j}\rangle_{i\in \Lambda_j}$. We say a family of $R$-modules  $\{M_j\}_{j\in \Gamma}$  is \emph{$u$-$S$-generated} by  $\{\{m_{i,j}\}_{i\in \Lambda_j}\}_{j\in \Gamma}$ provided that there exists an element $s\in S$  such that $sM_j\subseteq N_j$ for each $j\in \Gamma$ .
It is well known that a ring $R$ is a von Neumann regular ring if and only if every $R$-module is flat, if and only if any principal (finitely generated) ideal is generated by an idempotent (see \cite[Theorem 3.6.3]{fk16}). Now we give an $S$-analogue of this result.

\begin{theorem}\label{s-vn-ext-char}
Let $R$ be a ring and $S$ a multiplicative subset of $R$. The following statements are equivalent:
\begin{enumerate}
\item  $R$ is a $u$-$S$-von Neumann regular ring;
\item    for any $R$-module $M$ and $N$, there exists  $s\in S$ such that $s\Tor_1^R(M,N)=0$;
\item  there exists  $s\in S$ such that $s\Tor_1^R(R/I,R/J)=0$    for any ideals $I$ and $J$ of $R$;
\item   there exists  $s\in S$ such that $s\Tor_1^R(R/I,R/J)=0$   for any $S$-finite ideals $I$ and $J$ of $R$;
\item   there exists  $s\in S$ such that $s\Tor_1^R(R/\langle a\rangle,R/\langle a\rangle)=0$    for any element $a\in R$;
\item   any $R$-module is $u$-$S$-flat;
\item  the class of all principal ideals of $R$ is $u$-$S$-generated by  idempotents;
\item  the class of all finitely generated ideals of $R$ is $u$-$S$-generated by  idempotents.
\end{enumerate}
\end{theorem}
\begin{proof}
$(1)\Leftrightarrow (5):$ It follows from the equivalences:  $s\Tor_1^R(R/\langle a\rangle,R/\langle a\rangle)=0$ if and only if $\frac{s\langle a\rangle}{\langle a^2\rangle}=0$, if and only if there exists  $r\in R$ such that $sa=ra^2$.

$(2)\Leftrightarrow (6)$, $(8)\Rightarrow (7)$  and $ (3)\Rightarrow (4)\Rightarrow (5):$ Trivial.

$(2)\Rightarrow (3)$: Set  $M=N=\bigoplus\limits_{I\lhd R} R/I$. Then $(3)$ holds naturally.

$(3)\Rightarrow (2):$ Suppose $M$ is generated by $\{m_i\mid i\in \Gamma\}$ and $N$ is generated by $\{n_i\mid i\in \Lambda\}$. Well-order $\Gamma$ and $\Lambda$.  Set $M_0=0$ and $M_\alpha=\langle m_i\mid i<\alpha\rangle$  for each $\alpha\leq \Gamma$.  Then $M$ have a continuous   filtration  $\{M_\alpha\mid \alpha\leq \Gamma \}$ with $M_{\alpha+1}/M_\alpha\cong R/I_{\alpha+1} $ and $I_{\alpha}=\Ann_R(m_{\alpha}+M_\alpha\cap Rm_{\alpha})$. Similarly  $N$ has a continuous   filtration  $\{N_\beta\mid \beta\leq \Lambda \}$ with $N_{\beta+1}/N_\beta\cong R/J_{\beta+1} $ and $J_{\beta}=\Ann_R(n_{\beta}+N_\beta\cap Rn_{\beta})$.  Since $s\Tor_1^R(R/I_{\alpha},R/J_{\beta})=0$ for each $\alpha\leq \Gamma$ and $\beta\leq \Lambda$, it is easy to verify $s\Tor_1^R(M,N)=0$ by transfinite induction on both positions of $M$ and $N$.

$(5)\Rightarrow (3):$  By \cite[Exercise 3.20]{fk16}, we have $s\Tor_1^R(R/I,R/J)=\frac{s(I\cap J)}{IJ}$ for any ideals $I$ and $J$ of $R$. So we just need to show $s(I\cap J)\subseteq IJ$. Let $a\in I\cap J$. Since $s\Tor_1^R(R/\langle a\rangle,R/\langle a\rangle)=\frac{s\langle a\rangle}{\langle a^2\rangle}=0$, it follows that $sa\in s\langle a\rangle \subseteq \langle a^2\rangle\subseteq IJ$. Thus $s\Tor_1^R(R/I,R/J)=0$.

$(1)\Rightarrow (7):$  Let $s$ be an element in $S$ such that $sa=ra^2$ for some $r\in R$ and any $a\in R$. Set $e=ra$. Then $se=e^2$ and  $e\in \langle a\rangle$. Since $sa=ea\in \langle e\rangle$, we have $s\langle a\rangle\subseteq \langle e\rangle\subseteq \langle a\rangle$ for any $a\in R$.

$(7)\Rightarrow (8):$ Let  $\{I_j=Ra_{1,j}+\cdots +Ra_{n_j,j}\mid j\in \Gamma\}$ be the family of all finitely generated ideals of $R$. By $(3)$, there exists an element $s\in S$ such that for each $j\in \Gamma$ and $i=1,..,n_j$ there is an  idempotent $e_{i,j}\in Ra_{i,j}$ such that $s\langle a_{i,j}\rangle\subseteq \langle e_{i,j}\rangle$. Set $J_j=Re_{1,j}+\cdots +Re_{n_j,j}$. Then $J_j$ is a sub-ideal of $I_j$  such that $sJ_j\subseteq  I_j \subseteq J_j$. Claim that $J_j$ is generated by an idempotent. Indeed, for any $x\in J_j$, we have $x=r_1e_1+\cdots +r_ne_n=r_1e^2_1+\cdots +r_ne^2_n\in J_j^2$. Thus $J_j^2=J_j$. Since $J_j$ is finitely generated, $J_j=\langle e_j\rangle$ for some idempotent $e_j\in I_j$ by \cite[Theorem 1.8.22]{fk16}. So $\{I_j\mid j\in \Gamma\}$ is  $u$-$S$-generated by $\{\{e_j\}\mid j\in \Gamma\}$.

$(7)\Rightarrow (1):$  There are an element $s\in S$ and  a family of idempotents $\{e_a\mid a\in R\}$ such that $s\langle a\rangle \subseteq \langle e_a\rangle \subseteq \langle a\rangle$ for any $a\in R$. Write $e_a=ba$ for some $b \in R$. Then $e_a=e_a^2=b^2a^2$. Thus $sa=ce_a=cb^2a^2$ for some $cb^2\in R$. So $R$ is $u$-$S$-von Neumann regular.
\end{proof}

\begin{corollary}\label{s-local-vn}
Let $R$ be a ring and $S$ a multiplicative subset of $R$. If  $R$ is a $u$-$S$-von Neumann regular ring, then $R_S$ is a von Neumann regular ring. Consequently, any $u$-$S$-von Neumann regular ring  is $c$-$S$-coherent.
\end{corollary}
\begin{proof}  It follows from  Proposition \ref{s-inj-ext}, Corollary \ref{c-s-coh} and Theorem \ref{s-vn-ext-char}.
\end{proof}

Note that a ring $R$ such that $R_S$ is von Neumann regular is not necessary $u$-$S$-von Neumann regular.
\begin{example}\label{S-vn-not}
Let $\mathbb{Z}$ be the ring of all integers,  $S=\mathbb{Z}\setminus\{0\}$. Then $\mathbb{Z}_S=\mathbb{Q}$ is a von Neumann regular ring.  Let $p$ be a prime in $\mathbb{Z}$ and $M=\mathbb{Z}_{(p)}/\mathbb{Z}$. Then  $\Tor_1^{\mathbb{Z}}(\mathbb{Q}/\mathbb{Z}, \mathbb{Z}_{(p)}/\mathbb{Z})\cong \mathbb{Z}_{(p)}/\mathbb{Z}$  by \cite[Chapter I, Lemma 6.2(b)]{FS01}. It is easy to verify that $n \mathbb{Z}_{(p)}/\mathbb{Z}\not=0$ for any $n\in S$. Thus $M$ is not  $u$-$S$-torsion, and so $\mathbb{Z}$ is  not a $u$-$S$-von Neumann regular ring.
\end{example}

\begin{corollary} Let $R$ be a ring. Let $S$ be a  multiplicative subset of $R$ consisting of finite  elements. Then $R$ is a $u$-$S$-von Neumann regular ring if and only if $R_S$ is a von Neumann regular ring.
\end{corollary}
\begin{proof} We just need to show that if $R_S$ is a von Neumann regular ring then $R$ is a $u$-$S$-von Neumann regular ring. Let $S=\{s_1,\cdots ,s_n\}$. Set $s=s_1\cdots s_n$. By Proposition \ref{s-inj-ext},  for any $a\in R$, there exists  $s_i\in S$ and $r_a\in R$ such that $s_ia=r_aa^2$. Thus $sa=ra^2$ for any $a\in R$ and some $r\in R$.
\end{proof}

Since every  flat module is  $u$-$S$-flat,  von Neumann regular rings  are $u$-$S$-von Neumann regular. The following result shows $u$-$S$-von Neumann regular rings  are always von Neumann regular provided $S$ is a regular multiplicative set, i.e., the multiplicative set $S$ is composed of non-zero-divisors.
\begin{proposition}\label{s-vn-vn} Let $R$ be a ring and $S$ a regular multiplicative subset of $R$.
 Then $R$ is $u$-$S$-von Neumann regular if and only if  $R$ is  von Neumann regular.
\end{proposition}
\begin{proof}  We just need to show if $R$ is  $u$-$S$-von Neumann regular, then $R$ is von Neumann regular. Suppose  $R$ is a $u$-$S$-von Neumann regular ring. Then there exists  $s\in S$ such that  for any $a\in R$ there exists  $r\in R$ satisfying $sa=ra^2$.  Taking $a=s^2$, we have $s^3=rs^4$. Since  $s$ is a non-zero-divisor of $R$,  we have $1=sr$. Thus $s$ is a unit. So for any $a\in R$ there exists  $r\in R$ such that $a=(s^{-1}r)a^2$. It follows that $R$ is a von Neumann regular ring.
\end{proof}

However, the condition that ``any element in $S$ is a non-zero-divisor'' in Proposition \ref{s-vn-vn} cannot be removed.  Let $R$ be any ring and $S$ a multiplicative subset of $R$ containing a nilpotent element. Then $R$ is a $u$-$S$-von Neumann regular ring. Indeed, let $s$ be a nilpotent element in $R$ with nilpotent index $n$. Then $0=s^n\in S$. Thus for any $a\in R$, we have $0a=0a^2=0$. So $R$ is $u$-$S$-von Neumann regular.  If the   multiplicative subset $S$ of $R$  does not contain $0$, the condition that ``any element in $S$ is a non-zero-divisor'' in Corollary \ref{s-vn-vn} also cannot be removed.

\begin{example}\label{exam-not-ut-1}
Let $T=\mathbb{Z}_2\times \mathbb{Z}_2$ be a semi-simple ring and $s=(1,0)\in T$. Then any element $a\in T$ satisfies $a^2=a$ and $2a=0$. Let $R=T[x]/\langle sx,x^2\rangle$ with $x$ the indeterminate  and $S=\{1,s\}$ be a multiplicative subset of $R$. Then $R$ is a $u$-$S$-von Neumann regular ring, but $R$ is not von Neumann regular. Indeed, let $r=a+b\overline{x}$ be any element in $R$, where $\overline{x}$ is the residual element of $x$ in $R$ and $a,b\in T$. Then  $sr=s(a+b\overline{x})=sa=sa^2=s(a^2+2ab\overline{x}+b^2\overline{x}^2)=s(a+b\overline{x})^2=sr^2$. Thus $R$ is  $u$-$S$-von Neumann regular. However, since $R$ is not reduced, $R$ is not von Neumann regular by \cite[Theorem 3.6.16(2), Exercise 3.48]{fk16}.
\end{example}

Let $\p$ be a prime ideal of $R$. We say a ring $R$ is a \emph{$u$-$\p$-von Neumann regular ring}  shortly provided  $R$ is an $u$-$(R\setminus\p)$-von Neumann regular ring. The final result gives a new local characterization of von Neumann regular rings.
\begin{proposition}\label{s-vn-loc-char}
Let $R$ be a ring. Then the following statements are equivalent:
 \begin{enumerate}
\item  $R$ is a von Neumann regular ring;
\item   $R$ is a $u$-$\p$-von Neumann regular ring for any $\p\in \Spec(R)$;
\item   $R$ is a $u$-$\m$-von Neumann regular ring for any $\m\in \Max(R)$.
 \end{enumerate}
\end{proposition}
\begin{proof} $(1)\Rightarrow (2):$  Let $F$ be an $R$-module and $\m\in \Max(R)$. Then $F$ is flat, and thus $u$-$\m$-flat. So $R$ is an $u$-$\m$-von Neumann regular ring.

 $(2)\Rightarrow (3):$  Trivial.

 $(3)\Rightarrow (1):$  Let $M$ be an $R$-module. Then $M$ is  $\m$-flat for any $\m\in \Max(R)$. Thus $M$ is flat by Proposition \ref{s-flat-loc-char}. So $R$ is a von Neumann regular ring.
\end{proof}

\begin{acknowledgement}\quad\\
The author was supported by  the National Natural Science Foundation of China (No. 12061001).
\end{acknowledgement}

\end{document}